\documentclass[12pt]{amsart}

\usepackage{amsmath, amssymb, amsfonts, amsthm,amscd}
\usepackage{mathrsfs,bbm}
\usepackage{txfonts}
\usepackage[all]{xy}
\usepackage{nicefrac}

\newcommand{\C}{\mathbb{C}}

\DeclareMathOperator{\Ad}{\mathrm{Ad}}
\DeclareMathOperator{\td}{\mathrm{td}}
\DeclareMathOperator{\ch}{\mathrm{ch}}

\DeclareMathOperator{\GL}{\mathrm{GL}}
\DeclareMathOperator{\SL}{\mathrm{SL}}
\DeclareMathOperator{\Sp}{\mathrm{Sp}}
\DeclareMathOperator{\SO}{\mathrm{SO}}

\DeclareMathOperator{\B}{\mathrm{B}}

\newtheorem{theorem}{Theorem}[section]
\newtheorem{proposition}[theorem]{Proposition}

\newtheorem{corollary}[theorem]{Corollary}
\newtheorem{example}[theorem]{Example}
\newtheorem{remark}[theorem]{Remark}

\newtheorem{lemma}[theorem]{Lemma}

\title[Character of Irreducible Representations Restricted to Finite Order Elements]
{ Character of Irreducible Representations Restricted to Finite Order Elements - An Asymptotic Formula}

\author{SHRAWAN KUMAR \& DIPENDRA PRASAD}
\date{}

\address{S. Kumar: Department of Mathematics, University of North Carolina, Chapel Hill, NC 27599-3250, USA}
\email{shrawan@email.unc.edu} 
\address{D. Prasad: Department of Mathematics, Indian Institute of Technology Bombay, Mumbai, India}
\email{prasad.dipendra@gmail.com}
\begin{document}
\maketitle{}

\noindent
{\bf Abstract:} Let $G$ be a connected reductive group over the complex numbers and let $T\subset G$ be a maximal torus. For any $t\in T$ of finite order and any irreducible representation $V(\lambda)$ of $G$ of highest weight $\lambda$, we determine the character 
$\ch(t, V(\lambda))$ by using the Lefschetz Trace Formula due to Atiyah-Singer and explicitly determining the connected components and their normal bundles
of the fixed point subvariety $(G/P)^t\subset G/P$  (for any parabolic subgroup $P$). This together with Wirtinger's theorem gives an asymptotic formula for  $\ch(t, V(n\lambda))$ when $n$ goes to infinity. 

\section{Introduction}

Let $(G,B,T)$ be a triple consisting of a connected reductive algebraic group $G$ over $\C$,  a Borel subgroup $B$
containing a maximal torus $T$. Their $\mathbb{C}$-points are denoted by the same symbols, i.e., $T=T(\mathbb{C})$ etc. Let $\lambda: T \rightarrow \C^\times$ be a dominant integral weight, and $\left(\pi_\lambda, V(\lambda)\right)$
the associated highest weight representation of $G$ with character $\Theta_\lambda: G \rightarrow \C$. The character
$\Theta_\lambda: G \rightarrow \C$ is a function of $\lambda$ as well as that of $g \in G$ which is determined by its restriction to $T$. In this paper, we
study the asymptotic behaviour of $\Theta_\lambda(t)$ as a function of $\lambda$ for a fixed $t \in T$ of finite order. For example, if $t=1$,
$\Theta_\lambda(1) = \dim \pi_\lambda$,  which by the Weyl dimension formula is a polynomial function in $\lambda$ of degree equal to $\dim(G/\B)$. At the other extreme, if $t \in T$ is regular (and of finite order), then  $\Theta_\lambda(t)$
is a piecewise constant function of $\lambda$, an assertion which the reader will immediately recognize as a consequence of the Weyl character formula. The aim of this paper is to prove, more generally, that if $t\in T$ is
an element of finite order  then $\Theta_\lambda(t)$ is a piecewise polynomial function in $\lambda$ of degree which is bounded above
by the dimension of the maximal unipotent subgroup of the centralizer $G^t=Z_{G}(t)$, a possibly disconnected reductive subgroup of $G$. Furthermore, in some situations, we prove that the bound is achieved (see the following corollary).

The paper is inspired by some works of the second author with other collaborators, cf. \cite{P}, \cite{NPP},
as well as \cite{AK} for classical groups, in which they calculate   $\Theta_\lambda(t)$
for $t \in T$, a power of the Coxeter conjugacy class (which is the unique regular conjugacy class
in the derived group of $G$
of minimal order in the adjoint group, equal to the Coxeter number of the corresponding Weyl group), and find that
$\Theta_\lambda(t)$ is either zero or is, up to a sign, the dimension of an irreducible representation of the identity component $Z_G(t)^o$ 
of the centralizer $Z_{G}(t)$.
The present paper is less precise on concrete character values, but gives an asymptotic formula for $\Theta_\lambda(t)$ for all $t\in T$ of finite order (see the following corollary).
Precise calculation of the character values at elements of order 2 (for classical groups and $G_2$) is made in a recent Ph. D. thesis  of Karmakar, see \cite{Ka}. 

Let $\mathcal{L}(\lambda)=\mathcal{L}_P(\lambda):=G\times^P\mathbbm{C}_{-\lambda}\rightarrow X_P$ be the homogeneous line bundle over $X_P=G/P$ associated to a character $\lambda$ of $P$, where $P\supset B$ is a (standard) parabolic subgroup. 

The present paper uses the Lefschetz Trace Formula due to Atiyah-Singer  for the action of $t \in T$ on a homogeneous line bundle over $G/B$, and more generally
on $X_P$ to calculate the character  $\Theta_\lambda(t)$. Specifically, our main theorem of this note is the following result (cf. Theorem \ref{thm2.1}).

Fix a dominant character $\lambda$ of $T$ and 
let $P=P_\lambda$ be the unique standard parabolic subgroup of $G$ such that $\mathcal{L}(\lambda)$ is an ample line bundle over $X_P$. 

\begin{theorem}
For any $t\in T$ of finite order and any integer $n\geqslant 0$, 
$$\Theta_{n\lambda}(t^{-1})=\sum_{v\in Y_t} D_v^t(n,\lambda),$$
where
$$D_v^t(n,\lambda)=
\sum_{k\geqslant 0} \int_{X_P^t(v)}\,\,\frac{c_1\left(\mathcal{L}(\lambda)_{|_{X_P^t(v)}}\right)^k\frac{n^k}{k!}e^{-nv\lambda}(t)\cdot \td(X_P^t(v))}{\prod_{\alpha\in(vR^+_P)\backslash R((G^t)^o)}\, [1-e^\alpha(t)e^{- c_1(\overline{\mathcal{L}}(\alpha)_{|_{X_P^t(v)}})}]},
$$
where $Y_t$ denotes the set of connected components of  $Z_{G}(t)$, 
$R^+_P$ is the set of roots of the unipotent radical of $P$, $\td$ is the Todd genus of the tangent bundle, the notation $Y_t$ and $X_P^t(v)$ are explained in Proposition \ref{pps1.1} and the notation $\overline{\mathcal{L}}(\alpha)$ is as in the identification \eqref{eq12}. 
\end{theorem}

This theorem takes a simpler form when $P=B$ since in this case each $X_P^t(v)$ is isomorphic with the full flag variety of $Z_G(t)^o$ (cf. Corollary \ref{cry2.4}).

\vskip1ex

Let $d_v$ be the complex dimension of $X_P^t(v)$ and let $r$ be the order of $t\in T$. Consider the function 
$$
\chi^t_\lambda:~\mathbbm{Z}_{\geqslant 0}\rightarrow \mathbbm{C},\quad n \mapsto \Theta_{n\lambda}(t^{-1}).
$$
As an immediate consequence of the above theorem and Wirtinger Theorem,  we obtain the following corollary (cf. Corollaries \ref{cry2.2} and  \ref{cry2.5}).

\begin {corollary}
For any fixed $0\leq p<r$, $\chi_{\lambda_{|\left\{n\equiv p(mod\,r)\right\}}}^t$ is a polynomial function in $n$ of degree
$\leq \max_{v\in Y_t}\left\{d_v\right\}$.

Moreover, the coefficient of $n^{d_v}$ in $D_v^t(n,\lambda)$ restricted to $\{ n\equiv p (mod\, r) \}$ is equal to 
\begin{equation}\label{eq14}
\frac{1}{d_v!}\int_{X_P^t(v)} c_1 \left(\mathcal{L}(\lambda)_{|_{X_P^t(v)}}\right)^{d_v} \frac{e^{-pv\lambda}(t)}{\prod_{\alpha\in \left(v R^+_P\right)\backslash R\left(\left(G^t\right)^o\right)}\left(1-e^{\alpha}(t)\right)}.
\end{equation}

In particular, if  there is a unique $v_o\in Y_t$ such that $d_{v_o}=  \max_{v\in Y_t}\left\{d_v\right\}$. Then, for any fixed $0\leq p< r$, 
$$(\chi^t_\lambda)_{|_{n\equiv p (mod \ r)}}$$
is a polynomial function of degree exactly equal to $d_{v_o}$.
\end{corollary}

As a consequence of the above corollary,  we obtain the following result (cf. Theorem \ref{thm3.1}).
\begin {theorem}
Assume  that $t\in T$ is of order 2. Then, the function $\left(\chi_{\lambda}^t\right)_{|\{ n\in 2\mathbbm{Z}_{\geq0}\}}$ is a polynomial function of degree exactly equal to $d:= max_{v\in Y_t}\left\{d_v\right\}$. 
\end {theorem}
\vskip1ex

To prove the above results,  we prove several assertions of
independent interest regarding the relationship of the fixed points of the action of $t \in T$ on $G/P$ with the flag
varieties of  the identity component $Z_{G}(t)^o$ of  $Z_{G}(t)$: each connected component of the fixed point set is a homogeneous space for  $Z_{G}(t)^o$ (cf. Section 2, especially Proposition \ref{pps1.1} and Corollary \ref{cry1.2}). In the case $P=B$, each connected component of the fixed point set $X^t$ is isomorphic with the full flag variety of $Z_G(t)^o$ and there are exactly $\frac{\# W}{\# W(Z_G(t)^o)}$ many connected components (cf. Corollary \ref{cry1.2} and Lemma \ref{lem1.5}). 

\vskip1ex
After we completed this work, G. Lusztig pointed out his article `Michael Atiyah and Representation Theory' in \cite{H}, where he mentioned an application of the Lefschetz trace formula due to Atiyah-Singer \cite{AS} to get an explicit formula for the character of any irreducible  representation $V$ of $G$ at an element $t$ of finite order of $G$ as a sum of contributions of the connected components of the fixed point set of $t$ on the flag manifold of $G$. In this note we determine the connected components explicitly and thus work out such an explicit formula for the character $\ch (t, V)$.
\vskip1ex

\noindent
{\bf Acknowledgements:} We thank G. Lusztig and M. Vergne for their comments. In particular, Remark \ref{rmk2.2} is in response to a question by M. Vergne in a private communication.

\section{Fixed points of a flag variety under the action of a subgroup of torus}

Let $T\subset B\subset P\subset G$ be a maximal torus, a Borel subgroup, a standard parabolic in a connected reductive group $G$ over $\mathbbm{C}$. Let $S\subset T$ be a subgroup (not necessarily connected) and let $X_P=G/P$ be the partial flag variety. Then, $S$ acts on $X_P$ via left multiplication. Let $X^S_P$ be the fixed subvariety. Then, $X_P^S$ is smooth. Let $W^P$ be the set of minimal coset representative in $W/W_P$, where $W$ is the Weyl group of $G$ and $W_P$ is the Weyl group of $P$ (which is by definition the Weyl group of its Levi component). Let $(G^S)^o$ be the identity component of the fixed subgroup $G^S$ under the conjugation action of $S$ on $G$. 

\begin {proposition}\label{pps1.1}
With the  notation as above, the set of connected components of $X_P^S$ is bijectively parameterized by 
$$
Y_S:=\left\{v\in W^P:~B^S\cdot vP/P=vP/P \right\}.
$$

For $v\in Y_S$, let $X_P^S(v)$ be the corresponding connected component of $X_P^S$. Then, the morphism
$$
\hat{\phi}_v:(G^S)^o\longrightarrow X^S_P, \quad g \mapsto g\cdot vP/P
$$
descends to a $(G^S)^o$-equivariant variety isomorphism
$$
\phi_v:(G^S)^o/\left((G^S)^o\cap(vPv^{-1})\right)\overset{\sim}{\longrightarrow} X_P^S(v).
$$
Observe that the Borel subgroup $B^S\subset (G^S)^o\cap (vPv^{-1})$, since $v\in Y_S$. Moreover, since $X_P^S$ is smooth, its connected components coincide with its irreducible components.
\end {proposition}

\begin {proof}
Take a connected component $\mathcal{Z}\subset X_P^S$. It is clearly $B^S$-stable and hence, by Borel Fixed Point Theorem, contains a $B^S$-fixed point $vP/P$ (for some $v\in W^P$) since a $B^S$-fixed point, in particular, is a $T$-fixed point. We claim that the map 
$$
\hat{\phi}_v:~(G^S)^o\rightarrow X_P^S, \quad g\mapsto g\cdot vP/P
$$
induces an isomorphism
$$
\phi_v:(G^S)^o/\left((G^S)^o\cap(vPv^{-1})\right)\overset{\sim}{\longrightarrow} \mathcal{Z}:
$$

By the Bruhat decomposition, we get 
$$
X_P=\sqcup_{w\in W^P}BwP/P, 
$$
and hence
\begin{equation}\label{eq1}
X_P^S =\sqcup_{w\in W^P}B^SwP/P.
\end{equation}
Since $vP/P\in \mathcal{Z}$, we get
$$
(G^S)^o\cdot vP/P\subset \mathcal{Z}.
$$

Moreover, the isotropy subgroup of $v$ in $(G^S)^o$ is clearly $(G^S)^o\cap (vPv^{-1})\supset B^S$, thus we get a $(G^S)^o$-equivariant embedding
\begin{equation}\label{eq2}
\phi_v:(G^S)^o/\left((G^S)^o\cap(vPv^{-1})\right) \hookrightarrow \mathcal{Z}.
\end{equation}

We now prove that any $(G^S)^o$-orbit in $X_P^S$ is closed: By  identity (\ref{eq1}), any $(G^S)^o$-orbit $A$ in $X_P^S$ is a certain union $\sqcup_{w\in\theta(A)}B^S wP/P$, for some subset $\theta(A)\subset W^P$. Take $v_o\in \theta(A)$ of smallest length. 
Then, we prove that 
\begin{equation}\label{eq3}
B^S\cdot v_oP/P =v_o P/P.
\end{equation}

For if (\ref{eq3}) were false, take $\mathcal{U}_{\alpha}\subset B^S \subset (G^S)^o$ 
such that 
$$
\mathcal{U}_\alpha\cdot v_o P/P \supsetneqq v_o P/P,
$$
where $\mathcal{U}_\alpha$ is the one-parameter additive group corresponding to a positive root $\alpha$. In particular 
\begin{equation} \label{eq4}
v_o^{-1}\mathcal{U}_\alpha v_o\nsubset P \mbox{ and hence } v_o^{-1}\mathcal{U}_{-\alpha} v_o\subset \mathcal{U}^P,
\end{equation}
 where $\mathcal{U}^P$ is the unipotent radical of $P$, which gives 
$$
\mathcal{U}_{-\alpha}v_oP/P=v_oP/P.
$$

By (\ref{eq4}), $v_o^{-1}\alpha$ is a negative root and hence 
\begin{equation}\label{eq5}
s_\alpha v_o< v_o\quad \mbox{ by [K, Lemma 1.3.13]},
\end{equation}
where $s_\alpha \in W$ is the reflection through $\alpha$. Since $\mathcal{U}_\alpha\subset B^S\subset (G^{S})^o$, then so is $\mathcal{U}_{-\alpha}\subset(G^S)^o$
and hence the subgroup $\langle \mathcal{U}_\alpha, \mathcal{U}_{-\alpha}\rangle\subset G$ generated by $\mathcal{U}_\alpha$ and $\mathcal{U}_{-\alpha}$ is contained in $(G^S)^o$. Thus, $s_\alpha v_oP/P \in (G^S)^o\cdot v_oP/P$.
This is a contradiction to the minimal choice of $v_o$ in the $(G^S)^o$-orbit. Thus (\ref{eq3}) is true, proving that any $(G^S)^o$-orbit in $X_P^S$ is closed. Further, by (\ref{eq1}), there are only finitely many $(G^S)^o$-orbits in $X_P^S$. 

We return to the embedding $\phi_v$ as in (\ref{eq2}). We now show that $\phi_v$ is surjective: Since $\mathcal{Z}$ is stable under the action of $(G^S)^o$, it is a (finite) union of $(G^S)^o$-orbits. But, each $(G^S)^o$-orbit is closed in $X_P^S$ and hence $\mathcal{Z}$ is a disjoint union of finitely many closed $(G^S)^o$-orbits. But, $\mathcal{Z}$ being connected, it is a single $(G^S)^o$-orbit, proving that $\phi_v$ is a bijective isomorphism. 

Consider the map
$$
\psi:Y_S\rightarrow \left\{\mbox{set of the connected components of } X_P^S\right\},
$$
which takes 
$$
v\in Y_S \mapsto (G^S) \cdot vP/P.
$$
By the above proof $(G^S)^o\cdot vP/P$ is indeed a connected component of $X_P^S$ and hence $\psi$ is well defined. Moreover, $\psi$ is injective since any compact homogeneous variety of a reductive connected algebraic group over $\mathbbm{C}$ has a unique fixed point under any Borel subgroup, which follows from the Bruhat decomposition. 

Further, as proved above, $\psi$ is surjective. This proves the proposition.  
\end {proof}

\begin {corollary}\label{cry1.2}
With the notation as in Proposition \ref{pps1.1}, assume that $P=B$. Then, each connected component of $X^S_B$ is isomorphic, as a $(G^S)^o$-variety, with $(G^S)^o/B^S$. In particular, each connected component of $X_B^S$ is of the same dimension. 
\end {corollary}
\begin {proof}
By the above proposition applied to the case of $P=B$, we get that any connected component $\mathcal{Z}$ of $X_B^S$ is isomorphic with $(G^S)^o/\left((G^S)^o\cap(vBv^{-1})\right)$, for some $v\in Y_S \subset W$. 
Moreover, $B^S\subset (G^S)^o\cap(vBv^{-1})$. 
But, $vBv^{-1}$ being a solvable group and $B^S$ being a maximal solvable subgroup of $(G^S)^o$, we get that $B^S=(G^S)^o\cap(vBv^{-1})$. This proves the Corollary. 
\end {proof}
\begin {example}\label{exm1.3}
{\rm For a general parabolic $P$, $X_P^S$ does not necessarily have components of the same dimension. For example, take $G=\GL_n(\mathbbm{C})$ and $P$ a maximal parabolic subgroup such that $X_P=\mathbbm{P}^{n-1}$. Take $S\subset T$ be the subgroup of order $2$ generated by the diagonal element $\left((+1)^m,(-1)^{n-m}\right)$, which is $+1$ along the first $m$ entries of the diagonal and $-1$ in the following  $(n-m)$ entries. 
 In this case,}
$$
X^S_P=\mathbbm{P}^{m-1}\sqcup \mathbbm{P}^{n-m-1}.
$$
{\rm This provides a counter example if $n\neq 2m$.}
\end {example}

\begin {remark}\label{rmk1.4}
    {\rm (1) Following the notation and assumptions of Proposition \ref{pps1.1},} 
$$
Y_S=\left\{ v\in W^P:~ \mathcal{U}_v^S=e \right\},
$$
{\rm where $\mathcal{U}_v\subset G$ is the unipotent group with Lie algebra} 
$$
\bigoplus_{\alpha\in R^+ \cap v R_P^-}  \mathfrak{g}_\alpha = \bigoplus_{\alpha\in R^+ \cap v R^-} \mathfrak{g}_\alpha, 
$$
{\rm where $R^+$ is the set of positive roots of $G$ and $R^-_P$ is the set of  roots of the opposite unipotent radical of $P$. 
This follows since}
$$
B^S v~P/P= \mathcal{U}_v^S\cdot vP/P.
$$
\vskip1ex

    {\rm (2) If $S$ is generated by a finite order $t\in T$ which is regular, i.e., $e^\alpha(t)\neq 1$ for any root $\alpha$ of $G$, then $B^S=T$ (and hence $(G^S)^o=T$), $Y_S=W^P$ and the connected components of $X_P^S$ are all points (use Proposition \ref{pps1.1}).} 
\end {remark}

\begin {lemma}\label{lem1.5}
Following Proposition \ref{pps1.1}, we get the following:
\begin{equation}\label{eq6}
\sum_{v\in Y_S} \# \left( W((G^S)^o)/W\left((G^S)^o\cap (vPv^{-1})\right)\right)=\# W^P,
\end{equation}
 where $W(H)$ is the Weyl group of $H$. 

In particular,
\begin{equation}\label{eq7}
\# Y_S\geq \frac{\# W^P}{\#W((G^S)^o)}.
\end{equation}
For $P=B$, in fact we have the equality 
\begin{equation}\label{eq8}
\#Y_S =\frac{\# W}{\#W((G^S)^o)}.
\end{equation}
\end {lemma}

\begin {proof}
The identity (\ref{eq6}) follows from Proposition \ref{pps1.1} since $wP/P\in X_P^S$ for any $w\in W^P$ and $\phi_v: (G^S)^o/\left((G^S)^o\cap vPv^{-1}\right)\overset{\sim}{\rightarrow}X_P^S(v)$ is a $(G^S)^o$-equivariant isomorphism for any $v\in Y_S$ and, moreover,  $\sqcup_{v\in Y_S}X_P^S(v)=X_P^S$. 

Inequality (\ref{eq7}), of course, follows immediately from identity (\ref{eq6}). 

To prove  identity (\ref{eq8}), observe that for $P=B$, $(G^S)^o\cap v Pv^{-1}=B^S$ for any $v\in Y_S$ (cf. proof of Corollary \ref{cry1.2}). 

\end {proof}

\section{Character of any representation at finite order elements- An asymptotic formula via Lefschetz Theorem}

We apply the Lefschetz Trace Formula due to Atiyah-Singer (cf. \cite[Theorem 4.6]{AS}) for the case of the complex manifold $X=X_P=G/P$, the automorphism of $X_P$ is given by the left multiplication of a finite order element $t\in T$ and the vector bundle is a homogeneous line bundle $\mathcal{L}(\lambda):=G\times^P\mathbbm{C}_{-\lambda}\rightarrow G/P$ associated to a character $\lambda$ of $P$. 

Fix a connected component $X_P^t(v)=X_P^S(v)$ (for $v\in Y_S$) as in Proposition \ref{pps1.1}, where $S\subset T$ is the finite subgroup generated by $t$. Then, the $S$-equivariant line bundle.
\begin{equation}\label{eq9}
\mathcal{L}(\lambda)_{|_{X_{P}^t(v)}} \approx e^{-v\lambda}_{|_{S}}\otimes \hat{\mathcal{L}}(\lambda)_{|_{X_{P}^t(v)}},
\end{equation}
where $\hat{\mathcal{L}}(\lambda)_{|_{X_{P}^t(v)}}$ is the same line bundle as $\mathcal{L}(\lambda)_{|_{X_{P}^t(v)}}$ but with the trivial action of $S$. 

We next determine the normal bundle $N^t_v$ of $X_P^t(v)$ in $X_P$:
\vskip1ex

First of all the tangent space $T_{v\,P/P}(X_P)$ is given by the derivative of the curves (at $z=0$) $\gamma_\alpha:\mathbbm{C}\rightarrow X_P,$ $z \mapsto v\mbox{Exp}(z y_\alpha)P/P$, where $\alpha$ runs over the (negative) roots in $R_P^-$ (see Remark \ref{rmk1.4}(1)) and $y_\alpha$ is a fixed root vector of the root space $\mathfrak{g}_{-\alpha}$. For any $s\in T$, the action of $T$ on the above curve is given by 
\begin{eqnarray*}
s v\mbox{Exp}(z y_\alpha) P/P&=&v\left(v^{-1}s v\right) \mbox{ Exp}(z y_\alpha)(v^{-1}s^{-1} v)P/P\\
&=& v \mbox{ Exp}(\Ad (v^{-1}s v)\cdot z y_\alpha) P/P.
\end{eqnarray*}

From this we see that the derivative $\dot{\gamma}_\alpha(0)$ is transformed by the $T$-action via $v\cdot y_\alpha$. Thus, the tangent space (as  a $T$-module) is given by 
\begin{equation}\label{eq10}
T_v(X_P) = \bigoplus_{\alpha\in v R^-_P} \mathfrak{g}_{\alpha}.
\end{equation}

Similarly, considering the curve
$\mathbbm{C}\ni z \mapsto \mbox{Exp}(z  y_\alpha)vP/P$ in $X_P^t(v)$, we get that 
\begin{equation}\label{eq11}
T_v\left(X_P^t(v)\right)=\bigoplus_{\alpha\in(vR^-_P)\cap R((G^t)^o)}\, \mathfrak{g}_\alpha,
\end{equation}
where $R((G^t)^o)\subset R(G)$ denotes the set of all the roots of $(G^t)^o$ (here $G^t:=G^S$). 
Thus, the normal bundle $N^t_v$ over $X_P^t(v)$ is given by 
\begin{equation}\label{eq12}
N_v^t \approx \bigoplus_{\alpha\in (vR_P^+)\backslash R((G^t)^o)} \quad \overline{\mathcal{L}} (\alpha)_{|_{X_P^t(v)}},
\end{equation}
where $\overline{\mathcal{L}} (\alpha)_{|_{X_P^t(v)}}$ is the $(G^t)^o$-equivariant line bundle over $X_P^t(v)$ such that the fiber over $vP/P$ has $T$-weight $-\alpha$. Observe that 
\begin{equation}\label{eq13}
R^-((G^t)^o)\subset v\cdot R^-, \mbox { thus}\quad R^+((G^t)^o)\subset v\cdot R^+.
\end{equation}
If (\ref{eq13}) were false, take $\alpha\in R^-((G^t)^o)$ such that $v^{-1}\cdot \alpha \in R^+$. This gives $\alpha \in (v\cdot R^+)\cap R^-$. But, since $\mathcal{U}^t_v=(e)$ (see Remark \ref{rmk1.4}(1)), $t$ acts nontrivially on $\mathfrak{g}_\alpha$ and hence $\alpha\notin R((G^t)^o)$. This contradicts the choice of $\alpha$ and hence (\ref{eq13}) is proved. 

Let $V(\lambda)$ be the irreducible representation of $G$ with highest weight $\lambda\in \mathfrak{t}^*$ and let $t\in T$ be an element of finite order. Let $\ch(t, V(\lambda))$ be the trace of the action of $t$ on $V(\lambda)$.

We now apply \cite[Theorem 4.6]{AS} to get the following theorem. Let $P=P_\lambda$ be the unique standard parabolic subgroup of $G$ such that $\mathcal{L}(\lambda)$ is an ample line bundle over $X_P$. 

\begin{theorem}\label{thm2.1}
With the notation as above, for any integer $n\geqslant 0$, 
$$\ch\left(t^{-1},V(n\lambda)\right)=\sum_{v\in Y_S} D^t_v(n,\lambda),
$$
where
$$ D^t_v(n,\lambda):= \sum_{k\geqslant 0} \int_{X_P^t(v)}\,\,\frac{c_1\left(\mathcal{L}(\lambda)_{|_{X_P^t(v)}}\right)^k\frac{n^k}{k!}e^{-nv\lambda}(t)\cdot \td(X_P^t(v))}{\prod_{\alpha\in(vR^+_P)\backslash R((G^t)^o)}\, [1-e^\alpha(t)e^{- c_1(\overline{\mathcal{L}}(\alpha)_{|_{X_P^t(v)}})}]},
$$
as earlier $R^+_P$ is the set of roots of the unipotent radical of $P$ and $\td$ is the Todd genus of the tangent bundle, 
\end {theorem}
\begin {proof}
Use \cite[Theorem 4.6]{AS} together with the identities (\ref{eq6}) and (\ref{eq12}) for the line bundle $\mathcal{L}(n\lambda)$ over $X_P$. 
By the Borel-Weil-Bott Theorem, 
\begin{eqnarray*}
&&H^i\left(X_P,\mathcal{L}(n\lambda)\right)=0,\quad \mbox{for all } i>0\\
\mbox{and} && H^0\left(X_P,\mathcal{L}(n\lambda)\right)= V(n\lambda)^*. 
\end{eqnarray*}
Thus, 
\begin{eqnarray*}
\sum_{i\geq0}(-1)^i \mbox{Trace}\left(t,~H^i\left(X_P,\mathcal{L}(n\lambda)\right)\right)&=& \mbox{ch}\left(t, V(n\lambda)^*\right)\\
&=& \mbox{ch}\left(t^{-1}, V(n\lambda)\right).
\end{eqnarray*}
\end {proof}
\begin{remark} \label{rmk2.2} {\rm The above theorem remains true (by the same proof) for any standard parabolic subgroup $P$ replacing $P_\lambda$ as long as $\lambda$ extends to a character of $P$; in particular, for $P=B$. Applying the above theorem for $P=B$ and $n=1$,  we get that 
$\ch\left(t^{-1},V(\lambda)\right)$ is a piecewise polynomial function in $\lambda$. In fact, it is a polynomial function restricted to the dominant elements of any coset $X(T)/d\cdot X(T)$ ($d$ being the order of $t$), where $X(T)$ is the character group of $T$.}
\end{remark}
\vskip1ex

Let $d_v$ be the complex dimension of $X_P^t(v)$ and let $r$ be the order of $t\in T$. Consider the function 
$$
\chi^t_\lambda:~\mathbbm{Z}_{\geqslant 0}\rightarrow \mathbbm{C},\quad n \mapsto \mbox{ch}\left(t^{-1},V(n\lambda)\right).
$$

As a corollary of Theorem (\ref{thm2.1}), we get the following. 
\begin {corollary}\label{cry2.2}
For any fixed $0\leq p<r$, $\chi_{\lambda_{|\left\{n\equiv p(mod\,r)\right\}}}^t$ is a polynomial function in $n$ of degree
$\leq \max_{v\in Y_S}\left\{d_v\right\}$.

Moreover, the coefficient of $n^{d_v}$ in $D_v^t(n,\lambda)$ restricted to $\{ n\equiv p (mod\, r) \}$ is equal to 
\begin{equation}\label{eq14}
\frac{1}{d_v!}\int_{X_P^t(v)} c_1 \left(\mathcal{L}(\lambda)_{|_{X_P^t(v)}}\right)^{d_v} \frac{e^{-pv\lambda}(t)}{\prod_{\alpha\in \left(v R^+_P\right)\backslash R\left(\left(G^t\right)^o\right)}\left(1-e^{\alpha}(t)\right)}.
\end{equation}
\end{corollary}

\begin {proof}
The polynomial behavior  of $\chi_{\lambda|\left\{n\equiv p (mod r)\right\}}$ follows immediately from Theorem \ref{thm2.1}. To prove (\ref{eq14}), again use Theorem \ref{thm2.1} together with the fact that the constant term of the Todd genus of any manifold is 1 (cf.  \cite[Example 3.2.4]{F}). 
\end{proof}

\begin {remark}\label{rmk2.3}
{\rm Since $\mathcal{L}(\lambda)$ is an ample line bundle over $X_P=G/P$, by Wirtinger theorem (cf. \cite[Chapter 0, $\S$2]{GH}), for any $v\in Y_S$,} 
$$
\int_{X_P^t(v)}c_1\left(\mathcal{L}(\lambda)|_{X_P^t(v)}\right)^{d_v}>0.
$$
\end {remark}

As a corollary of Theorem \ref{thm2.1}, Corollary \ref{cry2.2} and Remark \ref{rmk2.3}, we immediately get the following result.
\begin{corollary} \label{cry2.5} Following the notation and assumptions as in Theorem \ref{thm2.1}, assume further that there is a unique $v_o\in Y_S$ such that $d_{v_o}=  \max_{v\in Y_S}\left\{d_v\right\}$. Then, for any fixed $0\leq p< r$, 
	$$\ch\left(t^{-1}, V(n\lambda)_{|_{n\equiv p (mod \ r)}}\right)$$
	is a polynomial function of degree exactly equal to $d_{v_o}$.
\end{corollary}

When $P=B$, i.e., $\lambda$ is a dominant regular highest weight, then Theorem \ref{thm2.1} specializes to the following.

\begin {corollary}\label{cry2.4}
With the notation and assumptions as in Theorem \ref{thm2.1} and with the additional assumption that $P=B$,
$$
\ch\left(t^{-1}, V(n\lambda)\right)=\sum_{v\in Y_S}\sum_{k\geqslant 0}\,
\int_{\overset{o}{X}^t} \frac{c_1\left(\mathcal{L}(v\lambda)_{|_{\overset{o}{X}^t}}\right)^k \frac{n^k}{k!} e^{-nv\lambda} (t)\td \overset{o}{X}^t}
{\prod_{\alpha\in(vR^+)\backslash \overset{o}{R}^t}\left[ 1-e ^\alpha (t) \cdot e^{-c_1 \left( \mathcal{L}(\alpha)_{ |_{\overset{o}{X}^t}}\right)}\right]},
$$
where $\overset{o}{X}^t:=(G^t)^o/B^t$ and $\overset{o}{R}^t$ denotes the set of roots of $(G^t)^o$. 
\end {corollary}

\begin {proof}
To deduce the corollary from Theorem \ref{thm2.1}, we need to observe that under the isomorphism of Proposition \ref{pps1.1} (for any $v\in Y_S$):
$$
\phi_v:~\overset{o}{X}^t\rightarrow X^t_B(v),\quad gB^t\mapsto gvB/B,
$$
the line bundle $\mathcal{L}(\mu)_{|_{X_B^t(v)}}$ pulls back to 
$$
\phi_v^* \left(\mathcal{L}(\mu)_{|_{X_B^t(v)}}\right)\simeq \mathcal{L}(v\mu)_{|_{\overset{o}{X}^t}}
$$
under the isomorphism
$$
\left[g,\mathbbm{1}_{-v\mu}\right] \mapsto\left[gv,\mathbbm{1}_{-\mu}\right],\quad \mbox{for } g\in(G^t)^o,
$$
where $\mathbbm{1}_{-v\mu}$ is a basis of $\mathbbm{C}_{-v\mu}$.
Moreover, 
$$
\phi_v^* \left(\overline{\mathcal{L}}(\mu)_{|_{X_B^t(v)}}\right)=\mathcal{L}(\mu)_{|_{\overset{o}{X}^t}}.
$$
\end {proof}
\begin {remark}\label{rmk2.5}
{\rm The Todd genus of any flag variety $G/P$ is determined by Brion \cite[$\S$3]{B}.} 
\end {remark}
\begin {example}\label{exm2.6}
{\rm Let the assumption be as in Theorem \ref{thm2.1}. Assume further that $t\in T$ is a regular element (i.e., $e^\alpha(t)\neq 1$ for any root $\alpha$ of $G$) of finite order $r$. Then, the function 
$n\mapsto \mbox{ch}\left(t^{-1}, V(n\lambda)\right)$ is the constant function 1, when $n$ is restricted to $r\mathbbm{Z}_{\geq0}$. }
\end{example}

Even though it follows easily from the Weyl character formula, but we deduce it from Corollary \ref{cry2.2}:

Since $t$ is regular, $Y_S=W^P$ and 
$B^t=T$, thus $X_P^t(v)=\left\{vP/P\right\}$ for any $v\in W^P$. Thus, each $d_v=0$ and by Corollary \ref{cry2.2}, we get (for $n\in
r\mathbbm{Z}_{\geq0}$) 
\begin{eqnarray*}
\mbox{ch}\left(t^{-1},V(n\lambda)\right)&=& \sum_{v\in W^P} \frac{1}{\prod_{\alpha\in v R_P^+}\left(1-e^\alpha (t)\right)}\\
&=&\sum_{v\in W^P} \frac{1}{\prod_{\alpha\in \left(vR^+_P\right)\cap R^+}\left(1-e^\alpha(t)\right)\cdot \prod_{\alpha\in \left(vR^+\right)\cap R^-}\left(1-e^\alpha(t)\right)},\\
&& \mbox{since } v\left(R^+\backslash R^+_P\right) \subset R^+\\
&=& \sum_{v\in W^P} (-1)^{\ell(v)} \frac{\prod_{\alpha\in vR^- \cap R^+} \,e^\alpha(t)}{\prod_{\alpha\in \left(vR^+_P\right)\cap R^+}\left(1-e^\alpha(t)\right)\cdot \prod_{\alpha\in vR^-\cap R^+}\left(1-e^\alpha(t)\right)}\\
&=& \sum_{v\in W^P}(-1)^{\ell(v)} \frac{e^{\rho-v\rho}(t)}{\prod_{\alpha\in R^+\backslash v\left(R^+\backslash R^+_P\right)}\left(1-e^\alpha (t)\right)}\\
&=& 1,\mbox{ by the parabolic analogue of the Weyl denominator formula.}
\end{eqnarray*}

\section{Specialization of results for an involution $t$}
In this section we consider elements $t\in T$ of order 2. As a consequence of Corollary \ref{cry2.2} in the case of involution $t$, we get the following. 

\begin {theorem}\label{thm3.1}
Follow the notation and assumptions as in Corollary \ref{cry2.2} and assume further that $t\in T$ is of order 2. Then, the function $\chi_{\lambda|\left\{ n\in 2\mathbbm{Z}_{\geq0}\right\}}^t$ is a polynomial function of degree exactly equal to $d:= max_{v\in Y_S}\left\{d_v\right\}$. 
\end {theorem}
\begin {proof}
By Corollary  \ref{cry2.2}, the coefficient of $n^{d_v}$ in $D_v^t(n,\lambda)$ restricted to $2\mathbbm{Z}_{\geq0}$ is equal to 
$$
\frac{1}{d_v!}\int_{X^t_P(v)}  c_1\left( \mathcal{L}(\lambda)_{|_{X^t_P(v)}}\right)^{d_v}\cdot \frac{1}{\prod_{\alpha\in\left(vR^+_P\right)\backslash R\left(\left(G^t\right)^o\right)} \left(1-e^\alpha(t)\right)}.
$$
But since $t$ is of order $2$, $e^\alpha(t)=\pm1$. 
Moreover, for $\alpha\notin R\left(\left(G^t\right)^o\right)$, $e^\alpha(t)=-1$. 
Thus, the above sum reduces to 
$$
\frac{1}{d_v!}\int_{X_P^t(t)} \,c_1\left(\mathcal{L}(\lambda)_{|_{X^t_P(v)}}\right)^{d_v}\cdot 2^{ -\# \left(vR_P^+\backslash R\left(\left(G^t\right)^o\right)\right)}.
$$
By Wirtinger’s Theorem (cf. Remark 2.3), the integral $\int_{X_P^t(v)} c_1 \left(\mathcal{L}(\lambda)_{|_{X^t_P(v)}}\right)^{d_v}>0$. 

Since, by Theorem \ref{thm2.1},
$$
\chi_\lambda^t (n) =\sum_{v\in Y_S}  D^t_v\left(n, \lambda\right),
$$
the theorem follows. 
\end {proof}
\begin{example}\label{exm3.2} Consider $G=\GL_n(\mathbbm{C})$ and $t=\left((+1)^m,(-1)^{n-m}\right)$ as in Example \ref{exm1.3}. In the case $n=2m$ and $m$ odd, using Theorem \ref{thm2.1}, it can be seen that for 
	$$\lambda:= \left(\lambda_1=\lambda_2=\dots =\lambda_m\geq \lambda_{m+1}=\lambda_{m+2}=\dots =\lambda_n\right)$$
	 with $\lambda_1$ and $\lambda_{m+1}$ of opposite parity, 
$$ \ch(t, V(\lambda))= 0.
$$
To prove this observe that, by Example \ref{exm1.3}, $X_P^t= \mathbbm{P}^{m-1}\sqcup \mathbbm{P}^{m-1}$. Moreover, in this case, 
$Y_S= \{1, v_o\}$, where $v_o$ is the cycle $(1, n, n-1, n-2, \dots, 2)$. Further, $e^{-\lambda}(t) = - e^{-v_o\lambda}(t)$.

A similar result can be obtained for $\Sp(2n)$ and $\SO(n)$.

\end{example}


\vskip9ex

\begin {thebibliography}{00}

\bibitem[AS]{AS} M. F. Atiyah and I. M. Singer, The index of elliptic operators III, Annals of Math., vol. {\bf 87} (1968), 546--604.

  \bibitem[AK]{AK} A. Ayyer and N. Kumari, {\it Factorization of Classical characters Twisted by Roots of Unity.} 
    J. Algebra {\bf 609} (2022), 437--483.

    \bibitem[B]{B} M. Brion, The push-forward and Todd class of flag bundles, Banach Center Publications, Warszawa, volume {\bf 36} (1996),
    45--50.

\bibitem[F]{F} W. Fulton, {\it Intersection Theory}, Second Edition, Springer (1998). 

\bibitem[GH]{GH} P. Griffiths and J. Harris, {\it Principles of Algebraic Geometry}, Wiley Interscience (1994). 

\bibitem[H]{H} N. Hitchin, Memories of Sir Michael Atiyah, Notices of AMS {\bf 66}  (2019), no. 11, 1834--1852.

\bibitem[Ka]{Ka} C. Karmakar, Computation of characters of classical groups at certain elements of order $2$, Preprint (2024). 

\bibitem[K]{K} S. Kumar, {\it Kac-Moody Groups, their Flag Varieties and Representation Theory}, Progress in Mathematics vol. {\bf 204}, Birkh\"auser (2002). 

\bibitem[NPP]{NPP} S. Nadimpalli, S. Pattanayak, and D.Prasad, {\it A product formula for the principal ${\SL}_2$}; preprint.
  
\bibitem[P]{P} D. Prasad, {\it A character relationship on $\text{GL}(n,\mathbb{C})$.} Israel J. Math. {\bf 211} (2016), no. 1, 257--270.


  
\end {thebibliography}
\end{document}